\documentclass[12pt, a4paper, reqno, oneside]{amsart}
\usepackage{amsmath,amssymb,latexsym,amsthm,color,url}
\usepackage[foot]{amsaddr}
\usepackage{cite, verbatim, longtable}

\newtheorem{theorem}{Theorem}
\numberwithin{theorem}{section}

\newtheorem{cor}[theorem]{Corollary}
\newtheorem{lemma}[theorem]{Lemma}

\numberwithin{equation}{section}

\def\AGmL{\operatorname{A\Gamma L}}
\def\GmL{\operatorname{\Gamma L}}
\def\AGL{\operatorname{AGL}}
\def\PGmL{\operatorname{P\Gamma L}}
\def\PSL{\operatorname{PSL}}

\def\SL{\operatorname{SL}}
\def\SU{\operatorname{SU}}
\def\Sp{\operatorname{Sp}}

\def\Sym{\operatorname{Sym}}
\def\Alt{\operatorname{Alt}}
\def\Aut{\operatorname{Aut}}
\def\Out{\operatorname{Out}}

\def\GL{\operatorname{GL}}

\def\ov{\overline}
\def\rk{\operatorname{rk}}

\def\32{\frac{3}{2}}
\def\12{\frac{1}{2}}
\def\2{{(2)}}

\title[Generating sets for primitive $\32$-transitive groups]{The minimal size of a generating set for primitive $\32$-transitive groups}
\author[Churikov]{Dmitry V. Churikov$^{1}$}
\email{churikovdv@gmail.com}
\author[Vasil'ev]{Andrey V. Vasil'ev$^{1}$}
\email{vasand@math.nsc.ru}
\author[Zvezdina]{Maria A. Zvezdina$^{2,1}$}
\address{$^1$Sobolev Institute of Mathematics, Koptyuga 4, Novosibirsk 630090, Russia}
\address{$^2$Novosibirsk State University, Pirogova 1, Novosibirsk 630090, Russia}
\email{zma@math.nsc.ru}

\begin{document}

{\let\thefootnote\relax\footnote{{D.V. Churikov and A.V. Vasil'ev were supported by Russian Scientific Foundation (project N 19-11-00039).}}}

\begin{abstract}
We refer to $d(G)$ as the minimal cardinality of a generating set of a finite group $G$, and say that $G$ is $d$-generated if $d(G)\leq d$. A transitive permutation group $G$ is called $\32$-transitive if a point stabilizer $G_\alpha$ is nontrivial and its orbits distinct from $\{\alpha\}$ are of the same size. We prove that $d(G)\leq4$ for every primitive $\32$-transitive permutation group~$G$, moreover, $G$ is $2$-generated except for the very particular solvable affine groups that we completely describe. In particular, all finite $2$-transitive and $2$-homogeneous groups are $2$-generated. We also show that every finite group whose abelian subgroups are cyclic is $2$-generated, and so is every Frobenius complement.\smallskip

{\sc Keywords:}  minimal generating set of a group, primitive permutation groups, $\32$-transitive groups, $2$-transitive groups, $2$-homogeneous groups, Frobenius complements.\smallskip

{\sc MSC: 20B05, 20B15, 20B20}
 \end{abstract}
\maketitle
\section{Introduction}\label{intro}

Let $G$ be a finite group. We refer to $d(G)$ as the minimal cardinality of a generating set of $G$ and say that $G$ is $d$-\emph{generated} if $d(G)\leq d$. Though, in the class of all finite groups, the well-known sharp upper bound on $d(G)$ is $\log |G|$ (hereinafter, all the logarithms are binary), there are important subclasses of it where $d(G)\leq d$ for some constant $d$. For example, modulo the classification of finite simple groups (CFSG), $d(G)\leq 2$ for simple groups \cite[Theorem~B]{AsGur} and $d(G)\leq 3$ for almost simple groups \cite[Theorem~1]{95VolLuc}.

In the case of finite permutation groups, there are well-known bounds on the number of generators in terms of the degree $n$ of a group $G\leq\Sym(n)$. It was noted in~\cite[Lemma~5.2]{MN} that $d(G)\leq n/2$ for an arbitrary permutation group $G$ (except $G=\Sym(3)$ where $d(G)=2$) and this bound is clearly the best possible. There is a constant $c$ such that for every transitive group $G$,
$$
d(G)\leq\frac{c\,n}{\sqrt{\log n}}.
$$
This was proved for nilpotent groups in~\cite{KN}, for solvable groups in~\cite{BKR}, and for arbitrary transitive groups in~\cite[Theorem~1]{00LMMt}. Moreover, this bound is sharp up to a multiplicative constant $c$ due to~\cite{KN}, and the best possible $c=\sqrt{3}/2$ by~\cite{Tr}.

If $G$ is a primitive group, then the upper bound is
$$
d(G)\leq\frac{c\,\log n}{\sqrt{\log\log n}}
$$
for some constant~$c$. This was established in~\cite[Corollary~1.6]{PS} for solvable groups, and in~\cite[Theorem~C]{00LMMp} for all primitive groups. The bound is also sharp due to~\cite{DK}, as it was observed in~\cite{PS} right after Corollary~1.6.

The main purpose of these short notes is to find the minimal number of generators in the case of primitive $\32$-transitive groups. Recall that a finite permutation group $G\leq\Sym(\Omega)$ is called $\12$-\emph{transitive} if all its orbits are of the same size greater than $1$ (the group of degree 1 is also considered as $\12$-transitive), and $\32$-\emph{transitive} if it is transitive and a point stabilizer $G_\alpha$ is $\12$-transitive on $\Omega\setminus\{\alpha\}$. Examples of $\32$-transitive groups are $2$-transitive groups and Frobenius groups. Since a nontrivial normal subgroup of a transitive group must be $\12$-transitive~{\cite[Theorem~10.3]{64Wiel}, normal subgroups of $2$-transitive groups also provide examples of $\32$-transitive groups. The notion of a $\32$-transitive group was introduced by Wielandt \cite[\S~10]{64Wiel}, who laid the foundations of a theory of such groups by proving that every imprimitive $\32$-transitive group must be Frobenius. Passman classified solvable groups from this class \cite{Pass67, Pass67a, Pass69}. Almost simple $\32$-transitive groups were described in~\cite{Bamb}, and the final step towards the classification of $\32$-transitive groups was done recently in~\cite{Gud,14LPS}.

\begin{theorem}\label{main}
Let $G$ be a primitive $\32$-transitive permutation group. If $d(G)>2$, then $G=VG_0$ is a solvable affine group with the zero vector stabilizer $G_0\leq\GL(V)$ which is imprimitive as a linear group, and up to permutation isomorphism one of the following holds:
\begin{enumerate}
\item $d(G)=3$ and $G_0$ is the two-dimensional group of monomial matrices with determinant equal to~$\pm1$ over any field of order $q\equiv1\pmod4;$
\item $d(G)=4$, $G=G_2$ in Table~\emph{\ref{tab:1}} of Section~\emph{\ref{s:32small}}, and $G_0$ is a four-dimensional group over the field of order $3$.
\end{enumerate}
\end{theorem}


As we will see, all the exceptions from Items~(i) and~(ii) of Theorem~\ref{main} are of rank greater than 2, so the following fact (nowhere mentioned, as far as we are aware) takes place.

\begin{cor}\label{2tran}
Every finite $2$-transitive permutation group is $2$-generated.
\end{cor}

The $2$-homogeneous but not $2$-transitive groups were described in~\cite[Proposition~3.1]{69Kan}. They are subgroups of $\AGmL_1(q)$, so the next assertion follows.

\begin{cor}\label{2hom}
Every finite $2$-homogeneous permutation group is $2$-generated.
\end{cor}

As observed above, a normal subgroup of a $2$-transitive group is $\32$-transitive, so Theorem~\ref{main} allows to describe all situations when a normal primitive subgroup of a $2$-transitive group is not $2$-generated.

\begin{cor}\label{sub2tran}
Let $G$ be a primitive normal subgroup of a $2$-transitive permutation group~$T$. Then either $d(G)\leq2$ or up to permutation isomorphism one of the following holds:
\begin{enumerate}
\item $G=G_1$ from Table~$\ref{tab:1}$ and $T$ is any subgroup of the group $M_1$ from Table~$\ref{tab:2}$ of Section~\emph{\ref{s:32small}}, containing $G$ and such that the index $|T:G|$ is a multiple of~$3;$
\item $G=G_2$ from Table~$\ref{tab:1}$ and $T$ is any subgroup of the group $M_2$ from Table~$\ref{tab:2}$ of Section~\emph{\ref{s:32small}}, containing $G$ and such that the index $|T:G|$ is a multiple of~$5$.
\end{enumerate}
\end{cor}

The important subclass of $\32$-transitive groups is the class of Frobenius groups. Let $G=KH$ be a Frobenius group with kernel $K$ and complement $H$. If $G$ is primitive, then $d(G)=\max\{2, d(H)\}$, so it suffices to find the number of generators of~$H$ (see details in Section~\ref{s:min}). If $H$ is solvable, then $H$ is $2$-generated due to Gruenberg's result~\cite[Lemma~9.9]{KWG} (in fact, he proved that $d(H)\leq2$ for every group $H$ whose abelian subgroups are cyclic, the property any Frobenius group enjoys). Unfortunately, we were unable to find the same assertions in the general case, so we proved them here (see Theorem~\ref{abeliancyclic} and Corollary~\ref{Frob} below).

Concluding the introduction, let us briefly discuss the number of generators of an imprimitive $\32$-transitive group $G$ which, by aforementioned Wielandt's result, is Frobenius. If $G$ is solvable, then Gruenberg proved that $d(G)=d_H(K)+1$, where $d_H(K)$ is the minimum number of generators of $K$ as $H$-group \cite[Proposition~9.5]{KWG}. We do not know if this is true in the general case. However, finding the number $d_H(K)$ does not seem easier than $d(G)$ itself. From our point of view, the best way to find $d(G)$ for a given imprimitive Frobenius group is to reduce the problem to the case when the kernel $K$ is an elementary abelian group and then apply \cite[Theorem~C]{AsGur}.

\section{On minimal generating sets of finite groups}\label{s:min}

Our main tool from finite group theory is the following general result.

\begin{theorem}\label{LM}\emph{\cite{97LucMen}} If $G$ is a noncyclic finite group with a unique minimal normal subgroup~$N$, then $d(G)=\max\{2,d(G/N)\}$.
\end{theorem}

The nonabelian simple groups are $2$-generated, so applying Theorem~\ref{LM} to the case when $N$ is a nonabelian simple group, we obtain that for an almost simple group $G$ with socle~$N$, $d(G)>2$ if and only if $d(G/N)>2$. Since the group $G/N$ is isomorphic to a subgroup of $\Out(N)$ which is known modulo CFSG for any simple group~$N$, one can easily find all almost simple groups $G$ with $d(G)>2$ (for each of them $d(G)=3$). The next lemma summarizes the information we need on such groups, see the corollary to~\cite[Theorem~1]{95VolLuc} and arguments after it.

\begin{lemma}\label{almost} Let $G$ be an almost simple group with the socle isomorphic to simple group~$S$. If $d(G)>2$, then one of the following holds:
\begin{enumerate}
\item $S$ is an orthogonal group of even dimension\emph{;}
\item $S$ is a linear group $\PSL_n(q)$ of even dimension $n$ and $G$ is not a subgroup of $\PGmL_n(q)$.
\end{enumerate}
\end{lemma}

The next lemma follows from \cite[Theorem~2]{95VolLuc}.

\begin{lemma}\label{normsimp}
If $G$ is a group containing a nonabelian simple group $S$ as a normal subgroup such that $G/S$ is solvable, then $d(G)=\max\{2,d(G/S)\}$.
\end{lemma}

Now our goal is to prove

\begin{theorem}\label{abeliancyclic} A finite group whose abelian subgroups are cyclic is $2$-generated.
\end{theorem}

Since a Frobenius complement does not contain noncyclic abelian subgroups (see, e.g., \cite[Theorem~18.1]{Pass-book}), the theorem has the following direct consequence.

\begin{cor}\label{Frob}
A finite Frobenius complement is $2$-generated.
\end{cor}

It is easily seen that every abelian subgroup of a finite group $G$ is cyclic if and only if Sylow $p$-subgroups of $G$ are cyclic for odd $p$ and cyclic or generalized quaternion for $p=2$. Finite groups whose all Sylow subgroups are cyclic were completely described by Zassenhaus in \cite{Zas} and are now called {\em $Z$-groups}.

\begin{lemma}\label{zgroups} Let $G$ be a $Z$-group. Then $G$ is generated by elements $a,b$ with $a^m=b^n=1$, $a^b=a^r$, where $(r-1,m)=(m,n)=1$ and $r^n\equiv1\pmod{m}$.
\end{lemma}

\begin{proof} See, e.g., \cite[Lemma~12.11]{Pass-book}.
\end{proof}

Gruenberg proved \cite[Lemma~9.9]{KWG} that a finite solvable group $G$ is $2$-generated if all odd Sylow subgroups of $G$ are cyclic and its Sylow $2$-subgroups are cyclic or dihedral or generalized quaternion. In the next lemma, we prove an even more general fact.

\begin{lemma}\label{solvable} Let $G$ be a solvable group whose odd Sylow subgroups are cyclic, and a Sylow $2$-subgroup includes a cyclic subgroup of index at most~$2$. Then $G$ is $2$-generated.
\end{lemma}

\begin{proof} We follow the line of the proof of \cite[Lemma~9.9]{KWG}. In view of this proof (and using the same notation), we may assume that a minimal normal subgroup $A$ of $G$ is an elementary abelian group of order $4$, a Sylow $2$-subgroup $T$ of $G$ includes $A$ as a proper subgroup and splits over it. The cases when $T$ is cyclic, quaternion or dihedral were considered in \cite[Lemma~9.9]{KWG}. Therefore, according to well-known Burnside's classification (see, e.g., \cite[Lemma~6.1.1]{Wolf} or \cite[Theorem~5.4.4]{Gor}), we may suppose that $T$ is either abelian of type $C_{2^{n-1}}\times C_2$, or the group $M_{2^n}$, or semidihedral group  $S_{2^n}$ (see, e.g., \cite[Section~5.4]{Gor} for definitions and elementary properties of these groups). However, in contrast to our assumptions, the latter group does not contain a normal elementary abelian group of order $4$, while the former two groups do not split over such a normal subgroup.
\end{proof}

A finite nonsolvable group whose abelian subgroups are cyclic were described by Suzuki \cite[Theorem~E]{Suz}. This allows us to prove easily the following lemma, thus completing the proof of Theorem~\ref{abeliancyclic}.

\begin{lemma}\label{nonsolvable}
A finite nonsolvable group whose abelian subgroups are cyclic is $2$-generated.
\end{lemma}

\begin{proof} Let $G$ be a finite nonsolvable group. Then, according to \cite[Theorem~E]{Suz}, the abelian subgroups of $G$ are cyclic if and only if $G$ includes a subgroup $G_0$ of index not exceeding~$2$, which is the direct product of the group $L=\SL_2(p)$, $p>3$ a prime, and a $Z$-group $Z$ of order coprime to the order of~$L$.

Take $Z=\langle a,b\rangle$ as in Lemma~\ref{zgroups} and observe that $L$ is generated by elements
$$
u=\left(\begin{array}{rr}
1 & 1\\
0 & 1
\end{array}\right)
\mbox{ and }
v=\left(\begin{array}{rr}
0 & 1\\
-1 & 0
\end{array}\right)
$$
of orders $p$ and $4$ respectively.

Since $|L|$ and $|Z|$ are coprime, the group $G_0$ is generated by elements $au$ and $bv$. Therefore, if $G=G_0$ we are done.

Let $|G:G_0|=2$. By \cite[Theorem~E]{Suz}, $G$ is generated by $G_0$ and an element $t$ such that $t$ normalizes $Z$ and $L$, $t^2=-e\in L$, where $e$ is the identity matrix, and $t$ acts on elements $a,b,u,v$ as follows:
$$
a^t=a^{-1},\quad b^t=b,\quad
u^t=\left(\begin{array}{rr}
1 & 0\\
-\omega & 1
\end{array}\right),
\mbox{ and }
v^t=\left(\begin{array}{rr}
0 & \omega^{-1}\\
-\omega & 0
\end{array}\right),
$$
where $\omega$ generates the multiplicative group of the underlying field of order~$p$. It is clear that $L=\langle u, u^t\rangle$. Thus, $G=\langle au, bt \rangle$, as required. \end{proof}

We conclude this section by finding the minimal number of generators for the group of two-dimensional monomial matrices with determinant equal to~$\pm1$ over a field having odd order. These groups appeared in Passman's classification of solvable $\12$-transitive linear groups \cite{Pass67,Pass67a,Pass69}. They are obviously intransitive on the set of nonzero vectors and imprimitive as linear groups (see, e.g., \cite{Pass67a}).

\begin{lemma}\label{Passman}
Let $q=p^m$, $p$ an odd prime. Set $G=S_0(q)$ to be the subgroup of $\GL_2(q)$ consisting of the matrices
$$
\left(\begin{array}{cc}
\alpha & 0\\
0 & \pm\alpha^{-1}
\end{array}\right)
\mbox{ and }
\left(\begin{array}{cc}
0 &  \alpha\\
\pm\alpha^{-1} & 0
\end{array}\right),\quad\alpha\in\mathbf{F}^*_q.
$$
Then $d(G)=2$ for $q\equiv3\pmod4$, and $d(G)=3$ for $q\equiv1\pmod4$.
\end{lemma}

\begin{proof} It is clear that $G$ is generated by the elements
$$
u=\left(\begin{array}{rr}
0 & 1\\
1 & 0
\end{array}\right),\quad
v=\left(\begin{array}{rr}
1 & 0\\
0 & -1
\end{array}\right),
\mbox{ and }
w=\left(\begin{array}{rr}
\omega & 0\\
0 & \omega^{-1}
\end{array}\right),\mbox{ where }\langle\omega\rangle=\mathbf{F}_q^*.
$$
Thus, $d(G)\leq3$. Put $K=\langle w^2\rangle$ and $\ov{G}=G/K$. Let also $\ov u, \ov v, \ov w$ be images of $u,v,w$ in~$\ov G$. Since $[u,w]\in K$, we see that $\overline u$ and $\overline w$ commute.

If $q\equiv1\pmod4$, then $[u,v]=-e\in K$, where $e$ is the identity matrix. It follows that $\ov G=\langle\ov u, \ov v, \ov w \rangle$ is an elementary abelian $2$-group of order~$8$, so $3\geq d(G)\geq d(\ov G)=3$ proving that $d(G)=3$.

Suppose that $q\equiv3\pmod4$. Then $-e\not\in K$, and $\ov G$ is a dihedral group of order $8$ generated by $\ov u$ and $\ov v$. In this case, $(vw)^{(q-1)/2}=-v$ and $(-v)^u=v$, so $G=\langle u, vw\rangle$ is $2$-generated. \end{proof}

\section{Primitive $\32$-transitive permutation groups: General case}\label{s:32proof}

As it was originally proved by Burnside \cite[\S~154, Theorem~XIII]{Burn}, the socle $S$ of a finite\linebreak $2$-transitive group $G$ is either a regular elementary abelian $p$-group, or a nonregular nonabelian simple group. In the former case, the group $G$ is said to be {\em affine} (a point stabilizer $G_0$ of $G$ acts on the socle as a linear group, that is $G_0$ is isomorphic to a subgroup of $\GL_m(p)$), while in the latter case we refer to $G$ as {\em almost simple}, because $S\leq G\leq\Aut(S)$ for some nonabelian simple group~$S$. It turns out that Burnside's result can be generalized to the case of the primitive $\32$-transitive groups.

\begin{lemma}\label{3/2Primitive}\emph{\cite[Theorem~1.1]{Bamb}}
A primitive $\32$-transitive group is either affine or almost simple.
\end{lemma}

The desired inequality $d(G)\leq2$ in the case of almost simple groups is a direct consequence of Lemma~\ref{almost} and the known description of almost simple $\32$-transitive groups, see \cite[Section~5]{Cam} or \cite[Table~7.4]{CamBook} for $2$-transitive groups and \cite[Theorem~1.2]{Bamb} otherwise.
\medskip

Thus, we may suppose that $G$ is affine. Then $G$ is permutation isomorphic to the semidirect product of a faithful $\mathbf{F}_p$-module $V$ of dimension $m$, where $n=p^m$ is the degree of $G$,  and the stabilizer $G_0$ of the zero vector in~$V$. Since $G$ is primitive, $G_0$ acts irreducibly on $V$, so $V$ is the unique minimal normal subgroup of $G$ and Theorem~\ref{LM} implies that $d(G)=\max\{2,d(G_0)\}$. Thus it suffices to bound $d(G_0)$. For convenience, we rename $G_0$ by $G$ in this case, so further $G$ is a subgroup of $\GL(V)\simeq\GL_m(p)$. Since the initial permutation group is assumed to be $\32$-transitive, the point stabilizer $G$ is $\12$-transitive on the set $V^*$ of nonzero elements of~$V$.

The classification of the $\12$-transitive linear groups was completed in \cite[Corollary 2]{14LPS} and gives us the following.

\begin{lemma}\label{12tranLin}
If $G\leq \GL(V)=\GL_m(p)$ is a $\frac{1}{2}$-transitive group on $V^*$, then one of the following holds:
\begin{enumerate}
\item $G$ is transitive on $V^*;$
\item $G$ is a Frobenius complement acting semiregularly on $V^*;$
\item $G=S_0(p^{m/2})$ with $p$ odd and $m$ even{\em;}
\item $G\leq \GmL_1(p^m);$
\item $\SL_2(5)\trianglelefteq G\leq \GL_2(p^{m/2})$, where $p^{m/2}=11$, $19$, $29;$
\item $\SL_2(5)\trianglelefteq G\leq \GmL_2(p^{m/2})$, where $p^{m/2}=169;$
\item $G\leq\GL_2(p)$ is solvable and $p=3,5,7,11,17;$
\item $G\leq\GL_4(3).$
\end{enumerate}
\end{lemma}

In Case (i) of Lemma~\ref{12tranLin}, the classification of the finite linear groups acting transitively on the set of nonzero vectors is due to C. Hering \cite{Her} (see also \cite[Appendix~1]{84Lieb} and \cite[Ch.~XII, Remark~7.5]{HB3}) and for our purposes can be summarized as follows.

\begin{lemma}\label{tranLin}
If $G\leq \GL(V)=\GL_m(p)$ is transitive on $V^*$, then one of the following holds:
\begin{enumerate}
\item $\SL_a(q)\trianglelefteq G\leq \GmL_a(q)$ with $q^a=p^m;$
\item $\Sp_{2a}(q)\trianglelefteq G\leq \GmL_{2a}(q)$ with $q^{2a}=p^m;$
\item $G_2(q)\trianglelefteq G\leq \GmL_6(q)$ with $q=2^m>2;$
\item $G\leq \GmL_1(p^m);$
\item $\SL_2(5)\trianglelefteq G\leq \GL_2(p^{m/2})$, where $p^{m/2}=11$, $19$, $29$, or $59;$
\item $\Alt(6)\simeq G\leq\GL_4(2)$, $\Alt(7)\simeq G\leq\GL_4(2)$, $\SL_2(13)\simeq G\leq\GL_6(3)$, $\SU_3(3)\simeq G\leq\GL_6(2);$
\item $G\leq\GL_2(p)$ is solvable and $p=3, 5, 7, 11, 23;$
\item $G\leq\GL_4(3).$
\end{enumerate}
\end{lemma}

The purpose of this section is to reduce the proof of Theorem~\ref{main} and Corollaries~\ref{2tran}--\ref{sub2tran} to direct computations in the small groups described in Items~(vii) and~(viii) of Lemmas~\ref{12tranLin} and~\ref{tranLin}.\smallskip

If $G$ is a Frobenius complement as in Item~(ii) of Lemma~\ref{12tranLin}, then we are done by Corollary~\ref{Frob}.\smallskip

Since every subgroup of a metacyclic group is $2$-generated and $\GmL_1(p^m)$ is metacyclic, $d(G)\leq2$ for $G$ in Item~(iv) of both Lemmas~\ref{12tranLin} and~\ref{tranLin}. \smallskip

If $G$ is from Item~(v) of Lemmas~\ref{12tranLin} and~\ref{tranLin}, then $G\leq ZR$, where $R=\SL_2(5)$ and $Z=\mathbf{F}_q^*$, $q=p^{m/2}$, is the multiplicative group of the underlying field (see Remark 2 after Theorem~1 in~ \cite{14LPS}). The generating set $\{u,v\}$ of $R$ can be chosen so that $u$ is of odd order coprime to~$|Z|$. If $a$ is a generator of $Z$, then the elements $au$ and $v$ generate $G$, as required.\smallskip

In Item~(vi) of Lemma~\ref{12tranLin}, $G$ is an extension by a field automorphism $t$ of order $2$ of the group $K=G\cap\GL_2(169)\leq Z_0R$, where $R=\SL_2(5)$ and $Z_0$ is a subgroup of order 28 in the multiplicative group $\mathbf{F}_{169}^*$ (see Remark 3 after Theorem~1 in~\cite{14LPS}). Since the group $\SL_2(13)$ does not include a subgroup isomorphic to $R$ (see, e.g., \cite[Table~8.2]{Bray}), the element $t$ does not centralize $R$. Hence it acts as the outer (diagonal) automorphism on $R$. Taking $u\in R$ as in the proof of Lemma~\ref{nonsolvable}, we have $R=\langle u, u^t\rangle$. If $a$ is a generator of $Z_0$, then $(|a|,|u|)=1$. Thus, $G=\langle au,t\rangle$, as required.\smallskip

In Item~(vi) of Lemma~\ref{tranLin}, $G=[G,G]$ and $G/Z(G)$ is nonabelian simple, so we are done.\smallskip

If $G$ is from Item~(i) of Lemma~\ref{tranLin}, then we apply Lemma~\ref{almost}.\smallskip

Let $R=\Sp_{2a}(q)$ be as in Item~(ii) of Lemma~\ref{tranLin}. Applying \cite[Corollary~2.10.4]{KM}, we see that for the normalizer $N=N_{\GL(V)}(R)$, the quotient $N/R$ is cyclic. Since $N=N_{\GmL(V)}(R)\cap\GL(V)$, it follows that $G/R$ lies in a metacyclic group. Note that the center $Z(G)=Z(R)$ is either trivial or the unique minimal normal subgroup of $G$. By Theorem~\ref{LM}, $d(G)=d(G/Z(G))$. Now Lemma~\ref{normsimp} yields $d(G)=d(G/S)=2$, as required.\smallskip

If $R=G_2(q)$ is from Item~(iii) of Lemma~\ref{tranLin}, then $G$ is a subgroup of an extension of $K=ZR$ by the group $T$ of field automorphisms, where $Z=\mathbf{F}_q^*$. Indeed, $R$ is a maximal subgroup of $\Sp_6(q)$ and the normalizer $N_{\Aut(\Sp_6(q))}(R)=RT$, see \cite[Table~8.29]{Bray}. Thus, $G/R$ is a subgroup of a metacyclic group and we apply Lemma~\ref{normsimp} once again.\smallskip

Minimal generating sets for the group $S_0(q)$, where $q$ is an odd prime power, from Item~(iii) of the Lemma~\ref{12tranLin} are given in Lemma~\ref{Passman}, so the conclusion of Theorem~\ref{main} follows for these groups. Let us prove that if $G=VG_0$, where $G_0=S_0(q)$, is a normal subgroup of a $2$-transitive group $T$, then $T$ is an affine group and the zero vector stabilizer in $T$ is from Items~(vii) and~(viii) of Lemma~\ref{tranLin}. Indeed, $V$ being the socle of $G$ is characteristic in it. Hence $V$ is a normal regular abelian subgroup of $T$, so $T$ is affine. Therefore, the zero vector stabilizer $T_0$ in $T$ is among the groups described in Lemma~\ref{tranLin}. Checking the items of Lemma~\ref{tranLin} one by one, we see that $T_0$ must be one of the groups from Items~(vii) and~(viii), as required.\smallskip

Thus, if a primitive $\32$-transitive permutation group $G$ is not $2$-generated, then either $G$ is described in Item~(i) of Theorem~\ref{main}, or $G$ is one of the groups from Items~(vii) and~(viii) of Lemmas~\ref{12tranLin} and~\ref{tranLin}. The latter is also true for any primitive normal subgroup $G$ of a $2$-transitive group~$T$.

\section{Primitive $\32$-transitive permutation groups: Exceptions}\label{s:32small}

We deal with the groups from Items~(vii) and~(viii) of Lemmas~\ref{12tranLin} and~\ref{tranLin} using $\operatorname{GAP}$~\cite{GAP} and its package $\operatorname{IRREDSOL}$ (see the log-file in~\cite{Logs}). First, we compute the set of all primitive $\32$-transitive solvable groups of degrees $3^2, 5^2, 7^2, 3^4, 11^2, 17^2, 23^2$. Using the function \verb"MinimalGeneratingSet(G)" for solvable groups, we filter out all $2$-generated groups. The resulting set consists of four affine groups $G_i$, $i=1,\ldots,4$, listed in Table~\ref{tab:1}. The table provides the degrees~$n$, ranks~$\operatorname{rk}$, sizes of minimal generating set $d$ of~$G$, the sizes~$|G_0|$ and matrix generators of the zero stabilizers~$G_0$, the parameters $(m,p,k,l)$ for a quick access to the exceptional groups through the $\operatorname{GAP}$ function \verb"PrimitiveSolvablePermGroup(m,p,k,l)". We verified that the groups $G_1$, $G_3$, and $G_4$ are permutation isomorphic to the groups (of corresponding degrees) described in Item~(i) of Theorem~\ref{main}. Thus, only the group $G_2$ must be mentioned separately (see Item~(ii) of Theorem~\ref{main}). It turns out that this group is not $2$-transitive and the zero stabilizer of this group is imprimitive as a linear group too.\smallskip

To complete the proof of Theorem~\ref{main} together with Corollaries~\ref{2tran} and~\ref{2hom}, it remains to consider nonsolvable groups of degree~$3^4$. We use the library of primitive permutation groups (built into GAP) to obtain all primitive $\32$-transitive nonsolvable groups of degree~$3^4$. Next we filter out groups with small generating set size equal to $2$. Mind that the function \verb"SmallGeneratingSet(G)" is apparently randomized, so another user can get another output on this step. In any case, there are only few groups left and they are also $2$-generated which can be checked by the function \verb"Is2Gen(G)" (see~\cite{Logs}).\smallskip

In order to prove Corollary~\ref{sub2tran}, we compute the normalizers of the groups from Table~\ref{tab:1} in the corresponding symmetric groups. It turns out that only the normalizers of $G_1$ and $G_2$ are $2$-transitive groups. Further computations show that if $G\in\{G_1,G_2\}$, then $M=N_{\Sym(n)}(G)=N_{\AGL(V)}(G)$, so $M=V\rtimes M_0$ is again an affine group and $M_0\leq\GL(V)$ is the stabilizer of zero vector  (in fact, this was already proved for $G_1$ in Section~\ref{s:32proof}). The group $G_0$ has $r-1$ nonzero orbits on $V$, where $r=\rk(G)$ ($r-1=3$ for $G=G_1$ and $r-1=5$ for $G=G_2$, see Table~\ref{tab:1}). Hence, if $G\trianglelefteq T\leq M$ and $T$ is $2$-transitive, then $|T:G|$ need to be a multiple of $r-1$. We verify that the latter condition is also sufficient, thus completing the proof. Indeed, our computations show that for $i=1,2$ and $G=G_i$, the group $\langle G_0, g \rangle$ acts transitively on $V^*$ for all $g\in M_0$ of order $r-1$. We collect the results on the maximal $2$-transitive groups $M$ with non-2-generated primitive normal subgroups in Table~\ref{tab:2} providing their degrees,  the sizes and matrix generators of the zero stabilizers $M_0$, and the information which of non-2-generated primitive normal subgroups from Table~\ref{tab:1} lies inside.

\medskip

\begin{table}[h]
\caption{Small primitive $\frac{3}{2}$-transitive groups $G$ with $d(G)>2$}\label{tab:1}
\begin{center}
\begin{tabular}[t]{|p{1em}|p{1.2em}|p{1.2em}|p{1.2em}|p{2em}|p{25em}|p{5em}|}
\hline
$G$&$n$&$\rk$&$d$&$|G_0|$& a minimal generating set of $G_0$&$(m,p,k,l)$ \\
\hline
$G_1$&$5^2$ & $4$&$3$&$16$&
$\tiny\begin{pmatrix} 0 & 4 \\ 4 & 0 \end{pmatrix}, \begin{pmatrix} 2 & 0 \\ 0 & 2 \end{pmatrix}, \begin{pmatrix} 4 & 0 \\ 0 & 1 \end{pmatrix}$&$(2,5,1,2)$ \\
\hline
$G_2$&$3^4$ & $6$&$4$ &$32$&
$\tiny\begin{pmatrix} 0 & 0 & 2 & 0 \\ 0 & 0 & 0 & 2 \\ 2 & 0 & 0 & 0 \\ 0 & 2 & 0 & 0\end{pmatrix}, \begin{pmatrix} 0 & 1 & 0 & 0 \\ 2 & 0 & 0 & 0 \\ 0 & 0 & 0 & 2 \\ 0 & 0 & 1 & 0\end{pmatrix},\begin{pmatrix} 2 & 1 & 0 & 0 \\ 1 & 1 & 0 & 0 \\ 0 & 0 & 2 & 1 \\ 0 & 0 & 1 & 1\end{pmatrix},\begin{pmatrix} 2 & 0 & 0 & 0 \\ 0 & 2 & 0 & 0 \\ 0 & 0 & 1 & 0 \\ 0 & 0 & 0 & 1\end{pmatrix}$&$(4,3,1,33)$ \\
\hline
$G_3$&$3^4$ & $6$&$3$ &$32$&
$\tiny\begin{pmatrix} 0 & 0 & 2 & 2 \\ 0 & 0 & 2 & 1 \\ 1 & 1 & 0 & 0 \\ 1 & 2 & 0 & 0\end{pmatrix}, \begin{pmatrix} 1 & 2 & 0 & 0 \\ 2 & 0 & 0 & 0 \\ 0 & 0 & 0 & 1 \\ 0 & 0 & 1 & 1\end{pmatrix},\begin{pmatrix} 2 & 0 & 0 & 0 \\ 0 & 2 & 0 & 0 \\ 0 & 0 & 1 & 0 \\ 0 & 0 & 0 & 1\end{pmatrix}$&$(4,3,2,6)$ \\
\hline
$G_4$&$17^2$ & $10$&$3$ &$64$&
$\tiny\begin{pmatrix} 0 & 8 \\ 15 & 0 \end{pmatrix}, \begin{pmatrix} 11 & 0 \\ 0 & 3 \end{pmatrix}, \begin{pmatrix} 16 & 0 \\ 0 & 1\end{pmatrix}$&$(2,17,1,15)$ \\
\hline
\end{tabular}
\end{center}
\end{table}

\begin{table}[h]
\caption{Maximal $2$-transitive groups $M$ with non-2-generated primitive normal subgroups}\label{tab:2}
\begin{center}
\begin{tabular}[t]{|p{1.2em}|p{1.2em}|p{2em}|p{15em}|p{4em}|}
\hline
$M$&$n$&$|M_0|$& a minimal generating set of $M_0$ & $G_i\trianglelefteq M$  \\
\hline
$M_1$&$5^2$&$96$&
$\tiny\begin{pmatrix} 0 & 2 \\ 1 & 0 \end{pmatrix}, \begin{pmatrix} 3 & 2 \\ 1 & 1 \end{pmatrix}$ & $G_1$\\
\hline
$M_2$&$3^4$&$3840$&
$\tiny\begin{pmatrix} 2 & 2 & 0 & 2 \\ 0 & 2 & 1 & 1 \\ 2 & 2 & 0 & 1 \\ 0 & 2 & 2 & 2\end{pmatrix}, \begin{pmatrix} 1 & 0 & 0 & 1 \\ 2 & 1 & 2 & 2 \\ 1 & 2 & 2 & 2 \\ 1 & 0 & 0 & 2\end{pmatrix}$ &$G_2$\\
\hline
\end{tabular}
\end{center}
\end{table}

\end{document}